\documentclass[12pt]{article}

\usepackage{amsmath}
\usepackage{amsfonts}
\usepackage{setspace}
\usepackage{enumerate}
\usepackage{epsfig}
\usepackage{subfigure}
\usepackage{rotating}
\usepackage{color}
\usepackage{multirow}
\allowdisplaybreaks \numberwithin{equation}{section}
\begin{document}
\bibliographystyle{plain}
\doublespacing

\newcommand{\ve}[1]{\mbox{\boldmath$#1$}}
\newcommand{\be}{\begin{equation}}
\newcommand{\ee}{\end{equation}}
\newcommand{\bc}{\begin{center}}
\newcommand{\ec}{\end{center}}
\newcommand{\bal}{\begin{align*}}
\newcommand{\enal}{\end{align*}}
\newcommand{\al}{\alpha}
\newcommand{\bt}{\beta}
\newcommand{\gm}{\gamma}
\newcommand{\de}{\delta}
\newcommand{\la}{\lambda}
\newcommand{\om}{\omega}
\newcommand{\Om}{\Omega}
\newcommand{\Gm}{\Gamma}
\newcommand{\De}{\Delta}
\newcommand{\Th}{\Theta}
\newcommand{\nno}{\nonumber}
\newtheorem{theorem}{Theorem}[section]
\newtheorem{lemma}{Lemma}[section]
\newtheorem{assum}{Assumption}[section]
\newtheorem{claim}{Claim}[section]
\newtheorem{proposition}{Proposition}[section]
\newtheorem{corollary}{Corollary}[section]
\newtheorem{definition}{Definition}[section]
\newtheorem{remark}{Remark}[section]
\newenvironment{proof}[1][Proof]{\begin{trivlist}
\item[\hskip \labelsep {\bfseries #1}]}{\end{trivlist}}
\newenvironment{proofclaim}[1][Proof of Claim]{\begin{trivlist}
\item[\hskip \labelsep {\bfseries #1}]}{\end{trivlist}}

\def \qed {\hfill \vrule height7pt width 5pt depth 0pt}
\def\refhg{\hangindent=20pt\hangafter=1}
\def\refmark{\par\vskip 2.50mm\noindent\refhg}
\def\Bbb#1{\mbox{\sf #1}}

\title{\textbf{Wavelet transform of Fractal \\ Interpolation Function}}
\author{Srijanani Anurag Prasad}
\date{}
\maketitle \vspace{-1.5cm}
{\singlespacing \bc
Stat-Math unit,
Indian Statistical Institute, Delhi Centre,\\
7, S. J. S. Sansanwal Marg, New Delhi, India. \\ Post Code:110016\\
janani@isid.ac.in
\ec }

\doublespacing

\begin{abstract}
In the present paper, the wavelet transform of Fractal Interpolation Function (FIF) is studied. The wavelet transform of FIF is obtained through two different methods. The first method uses the functional equation through which FIF is constructed. By this method, it is shown that the FIF belongs to Lipschitz class of order $\delta, (0<\de \leq 1),$ under certain conditions on free parameters. The second method is via Fourier transform of FIF. This approach gives the $\la$-regularity, $(0<\la)$, of FIF under certain conditions on free parameters. Fourier transform of a FIF is also derived in this paper to facilitate the approach of wavelet transform of a FIF via Fourier transform.
\end{abstract}

{\bf Key Words :} Fractal, Interpolation Function, Wavelet transform, Fourier transform, Functional equation,

{\bf Mathematics Subject Classification:} Primary 28A80, 41A05
\newpage

\section{Introduction}
Fractal Interpolation Function (FIF) was introduced by Barnsley~\cite{barnsley86} using the theory of Iterated Function System (IFS). It is a new and novel method to construct irregular functions from interpolation data. The construction of FIF depends on a functional operator which was investigated by Read~\cite{read51} and Bajraktarevic~\cite{bajraktarevic56}. It is now known as Read-Bajraktarevic operator and it defines a functional equation on FIF. In~\cite{girgensohn93,girgensohn94}, the nowhere differentiable functions are studied as solution of functional equations. In this paper, the wavelet transform of FIF is studied through two different methods. The first method uses the functional equation of FIF to find  bounds on the wavelet transform of FIF. By this method, we show that the FIF belongs to Lipschitz class of order $\delta$ under certain conditions. The second method is via Fourier transform of FIF. This approach gives the $\la$-regularity, $(0<\la)$, of FIF under certain conditions on free parameters used in the construction of FIF.

The present paper is organized as follows: In Section~\ref{sec:constr}, a brief introduction on construction of FIF is given. For convenience, some points pertaining to notation and terminology are also noted in the same section. Following the approach of~\cite{feng05}, the Fourier transform of a general FIF is obtained in Section~\ref{sec:ft}.  The result given in~\cite{feng05} follows as a special case. Section~\ref{sec:wt} is devoted to Wavelet transform of FIF. It is divided in two subsections - Subsection~\ref{sec:wt1}, wherein the Wavelet transform of FIF is obtained using the recursive functional equation and Subsection~\ref{sec:wt2}, wherein the Wavelet transform of FIF is obtained using the Fourier transform of FIF.

\section{Construction of FIF}\label{sec:constr}
In this section, a brief introduction on the construction of a Fractal Interpolation Function (FIF) is given. This section is based on~\cite{barnsley86,barnsley88}.

Given an interpolation data $\{(x_i,y_i) \in \mathbb{R}^2 :
i=0,1,\ldots,N \}$, where $-\infty < x_0 < x_1 < \ldots < x_N <
\infty$, the interval $[x_0,x_N]$ is denoted by $I$ and the smaller intervals $[x_{k-1},x_k]$ are denoted by $I_k$ for $k=1,2,\ldots,N$.  The functions $ L_k : I \rightarrow I_k $ and $ F_k : I \times \mathbb{R} \rightarrow
\mathbb{R}  $ for $k=1,2,\ldots,N$ are   defined as
\begin{align}\label{eq:LkFk}
L_k(x) &= a_k x + b_k = \frac{x_k-x_{k-1}}{x_N-x_0} \ x +  \frac{x_N
x_{k-1}-x_0 x_k}{x_N-x_0} \nno \\
F_k(x,y)&= \gm_k y + q_k(x)
\end{align}
respectively. In~\eqref{eq:LkFk},  $\gm_k$  are free variables chosen  such that $|\gm_k|
< 1$ and $ q_k$  are continuous functions chosen such that the following conditions on $F_k$ are satisfied:
\begin{align} \label{eq:Fncond}
|F_k(x,y)-F_k(\bar{x},\bar{y})| &\leq s d_M((x,y),(\bar{x},\bar{y})) , \ s < 1\nno \\
F_k(x_0,y_0)=y_{k-1} \quad &\mbox{and} \quad
F_k(x_N,y_N)=y_{k}
\end{align}
where, $d_M$ is a suitable metric equivalent to Euclidean metric in $\mathbb{R}^2$. Finally, the functions $ \om_k: I \times \mathbb{R} \rightarrow I \times \mathbb{R}$ are  defined as
\begin{equation}\label{eq:omk}
\om_k(x,y)=(L_k(x),F_k(x,y)).
\end{equation}
It is shown in~\cite{barnsley86,barnsley88} that $\om_k$ defined by~\eqref{eq:omk} are contraction maps with respect to a metric equivalent to Euclidean metric. Consequently,
\begin{align}\label{eq:ifs}
\{I \times \mathbb{R}; \om_k , k=1,2,\ldots, N \}
\end{align}
is a hyperbolic Iterated Function System (IFS) and therefore there exists an attractor $ A$ in $H(\mathbb{R}^2)$, the space of compact sets in $\mathbb{R}^2$, such that $ A = \bigcup\limits_{k=1}^N \om_k(A)$, where $\om_k(A)= \{\om_k(x,y) :\ (x,y) \in A \}$. It is proved in~\cite{barnsley86}, that attractor A of IFS~\eqref{eq:ifs} is graph of a continuous function $\ f : I \rightarrow
\mathbb{R} $ which interpolates the given data $ \{(x_i,y_i) : i=0,1,
\ldots ,N \} $,  i.e. $ A = \{ (x,f(x)) : x \in I \} $ and  $\
f(x_i)=y_i $ for $i=0,1, \ldots ,N$. The Fractal Interpolation Function (FIF) is thus defined as:

\begin{definition}(c.f.~\cite{barnsley86,barnsley88})
The \textbf{Fractal Interpolation Function (FIF)} for the
interpolation data $ \{(x_i,y_i) : i = 0,1,\ldots,N \}$ is defined
as the continuous function $\ f : I \rightarrow \mathbb{R} $, whose
graph is the attractor of IFS $ \{I \times \mathbb{R};\om_k,
k=1,2,\ldots,N \}$, where $\om_k$ are defined by~\eqref{eq:omk}.
\end{definition}


It is shown in~\cite{barnsley86,barnsley88} that the FIF $f$ is a fixed point of Read-Bajraktarevic operator $T$ defined by,
\begin{align*}
T(g)(x)  = F_k(L_k^{-1}(x),g(L_k^{-1}(x))), \ x \in I_k
\end{align*}
on the space $ ({\cal G}, d_{{\cal G}})$, where the set ${\cal G}$ is defined by $ {\cal G} = \{ g: \  g : I \rightarrow \mathbb{R} \ \mbox{is continuous},\\  g(x_0) = y_0 \ \mbox{and} \ g(x_N) = y_N\}$ and the maximum metric $d_{{\cal G}}$  is
given by $\\ d_{{\cal G}} (g,\hat{g}) = \max\limits_{x \in I } |g(x)~-~\hat{g}(x)|$,
$\ g,\hat{g} \in {\cal G}$. Hence, FIF $f$ satisfies
the functional equation
\begin{align}\label{eq:funceqn}
f(x) = \gm_k f(L_k^{-1}(x)) + q_k(L_k^{-1}(x)), \ x \in I_k  \ \mbox{and} \  k=1,2,\ldots,N.
\end{align}

For convenience, we assume that  $x_0=0, x_N=1$, $x_k -x_{k-1} =\frac{1}{N}$ for $k=1,\ldots,N$ and $y_0=y_N=0$ in the rest of our paper. So,
\begin{align}\label{eq:akbk}
a_k = \frac{1}{N} \quad \mbox{and} \quad b_k = \frac{k-1}{N}.
\end{align}
Also, for notational convenience, the value of an empty product is taken to be unity and the value of an empty sum is taken to be zero.

\section{Fourier Transform of a FIF}\label{sec:ft}

In this section, the Fourier transform of a FIF is obtained for a general $q_k$. The result given in~\cite{feng05} follows as a special case.

Let $f$ be a FIF obtained from an interpolation data. Extend $f$ to $\mathbb{R}$ by defining $f(x) = 0 $ if $ x \not \in I$ and hence $f \in L_1(\mathbb{R})$. The Fourier transform of $f$ is given by the following theorem:

\begin{theorem}
Let $f$ be a FIF obtained from the interpolation data $\{(x_i,y_i) \in \mathbb{R}^2 : i=0,1,\ldots,N \}$, where $0=x_0<x_1<\ldots<x_N=1$, $x_k -x_{k-1} =\frac{1}{N} $ for $k=1,2,\ldots,N$ and $y_0=y_N=0$. Extend $f$ to $\mathbb{R}$ by defining $f(x) = 0 $ if $ x \not \in I =[x_0,x_N]$. The Fourier transform of $f$ is given by
\begin{align}\label{eq:ftf}
\hat{f}(\om )
& =  \sum\limits_{j=1}^{\infty} \frac{1}{N^j} \sum\limits_{k_1,k_2,\ldots k_j=1}^N \gm_{k_1} \gm_{k_2}\ldots \gm_{k_{j-1}} e^{-i \om p_{k_1,k_2,\ldots,k_j}}  \int\limits_{I}  q_{k_j}(x) e^{-i \om \left(\frac{x}{N^{j}}\right)}  dx
\end{align}
where, $p_{k_1,k_2,\ldots k_j}=\frac{k_1-1}{N}+\frac{k_2-1}{N^2}+ \ldots + \frac{k_j-1}{N^j}$.
\end{theorem}
\begin{proof}
The Fourier transform of $f$ is
\begin{align*}
\hat{f}(\om) &= \int_I f(x) e^{-i \om x} dx
= \sum\limits_{k=1}^N \int\limits_{ I_k} f(x)  e^{-i \om x} dx .
\end{align*}
Using~\eqref{eq:funceqn} in the above equation,
\begin{align*}
\hat{f}(\om)&= \sum\limits_{k=1}^N \int\limits_{I_k} [\gm_k f(L_k^{-1}(x)) + q_k(L_k^{-1}(x))]  e^{-i \om x} dx
\end{align*}
Substituting $x=L_k(x)=\frac{x+k-1}{N}$, we have
\begin{align}\label{eq:temp1}
\hat{f}(\om)
& = \sum\limits_{k=1}^N \frac{1}{N} \int\limits_I [\gm_k f(x) + q_k(x)]  e^{-i \om (\frac{x+k-1}{N})} dx  \nno \\
& = \frac{1}{N} \sum\limits_{k=1}^N  e^{-i \om (k-1)/N} \bigg(\gm_k \hat{f}\left(\frac{\om}{N}\right) + \int\limits_{I}  q_k(x) e^{-i \om x/N }  dx \bigg).
\end{align}
Consider any $j=1,2,\ldots$. Replace $\om$ by $\frac{\om}{N^j}$ in~\eqref{eq:temp1}. Then,
\begin{align*}
\hat{f}\left(\frac{\om}{N^j}\right)
& = \frac{1}{N} \sum\limits_{k=1}^N  e^{-i \om (k-1)/N^{j+1}} \bigg(\gm_k \hat{f}\left(\frac{\om}{N^{j+1}}\right) + \int\limits_{I}  q_k(x) e^{-i \om x/N^{j+1}}   dx \bigg).
\end{align*}
Therefore, by induction, for any $n=1,2,\ldots,$,
\begin{align}\label{eq:temp2}
\hat{f}(\om)
& =  \frac{1}{N ^n }\sum\limits_{k_1,k_2,\ldots,k_n=1}^N  \gm_{k_1} \gm_{k_2}\ldots \gm_{k_{n}} e^{-i \om p_{k_1,k_2,\ldots,k_n}}
  \hat{f}\left(\frac{\om}{N^n}\right) \nno \\ & \quad \mbox{}+ \sum\limits_{j=1}^{n} \frac{1}{N^j} \sum\limits_{k_1,k_2,\ldots,k_j=1}^N \gm_{k_1} \gm_{k_2}\ldots \gm_{k_{j-1}} e^{-i \om p_{k_1,k_2,\ldots,k_j}}  \int\limits_{I}  q_{k_j}(x) e^{-i \om x/N^j}  dx.
\end{align}
With the argument similar to~\cite{feng05},  the Fourier transform of $f$ given by~\eqref{eq:ftf} is obtained from~\eqref{eq:temp2} as $n \rightarrow \infty$.
\qed
\end{proof}

\begin{corollary}
If $q_k, k=1,\ldots,N$ in~\eqref{eq:LkFk} are polynomials of degree $m_k$ i.e, $q_k(x)=\sum\limits_{r=0}^{m_k} c_{k,r} x^r, c_{k,m_k} \neq 0$, the Fourier transform of FIF $f$ is given by
\begin{align}\label{eq:cor}
\hat{f}(\om )
& = \sum\limits_{j=1}^{\infty}  \sum\limits_{k_1,k_2,\ldots,k_j=1}^N \gm_{k_1} \gm_{k_2}\ldots \gm_{k_{j-1}} e^{-i \om p_{k_1,k_2,\ldots,k_j}} \quad  \times \nno\\ & \quad \mbox{} \times \sum\limits_{r=1}^{m_{k_j}} c_{k_j,r}  \Bigg\{ \bigg[ \frac{r N^{j}}{ \om^2}- \frac{i r (r-1) N^{2j}}{\om^3} - \frac{r (r-1)(r-2)N^{3j}}{\om^4}+  \ldots \nno \\ & \quad \mbox{} - \frac{(-i)^{r+1} N^{rj} r!}{\om^{r+1}} \bigg]e^{-i \om/N^{j}}+ \frac{ (-i)^{r+1} N^{rj} r! }{\om^{r+1}}\Bigg\}.
\end{align}
\begin{proof}
Using $q_k(x)=\sum\limits_{r=0}^{m_k} c_{k,r} x^r $ for $k=1,2,\ldots,N$ in~\eqref{eq:ftf} and integrating, we have
\begin{align}\label{eq:cortemp}
\hat{f}(\om)
& = \sum\limits_{j=1}^{\infty} \frac{1}{N^j} \sum\limits_{k_1,k_2,\ldots,k_j=1}^N \gm_{k_1} \gm_{k_2}\ldots \gm_{k_{j-1}} e^{-i \om p_{k_1,k_2,\ldots,k_{j-1}}} \ \times \nno \\ & \quad \mbox{}  \times \sum\limits_{r=0}^{m_{k_j}} c_{k_j,r}\left[ \bigg\{ \frac{i N^j}{\om} + \frac{N^{2j} r}{\om^2} - \frac{i N^{3j} r(r-1)}{\om^3} - \frac{N^{4j} r(r-1)(r-2)}{\om^4}  + \ldots \right. \nno \\ & \left. \quad \mbox{} - \frac{(-i)^{r+1} N^{(r+1)j} r! }{\om^{r+1}}\bigg\} e^{-i \om k_j /N^j}  +  \frac{(-i)^{r+1} N^{(r+1)j} r!}{\om^{r+1}} e^{-i \om (k_j-1) /N^j} \right].
\end{align}

By~\eqref{eq:Fncond}, the constants $c_{k,r}, \ k=1,2,\ldots,N$ satisfy the following conditions:
\begin{itemize}
\item $c_{1,0} =0$
\item $\sum\limits_{r=0}^{m_k} c_{k,r} = c_{k+1,0}, k=1,\ldots,N-1$
\item $\sum\limits_{r=0}^{m_N} c_{N,r} = 0$.
\end{itemize}
Using the above conditions in~\eqref{eq:cortemp}, the Fourier transform of FIF $f$  as given by~\eqref{eq:cor} is obtained.
\qed
\end{proof}
\end{corollary}

\begin{remark}
If $q_k$ in~\eqref{eq:LkFk} are linear polynomials i.e, $q_k(x)=c_k x + d_k$,  the Fourier transform of the FIF given in~\cite{feng05} as
\begin{align}\label{eq:hatf}
\hat{f}(\om ) & = \frac{1}{\om^2} \sum\limits_{j=1}^{\infty} N^j(e^{-i \frac{\om}{N^j}}-1)
\left\{ \sum\limits_{k_1,k_2,\ldots,k_j=1}^N \gm_{k_0}\gm_{k_1} \ldots\gm_{k_{j-1}}c_{k_j} e^{-i \om p_{k_1,k_2,\ldots,k_j}} \right\}
\end{align}
follows from above corollary.
\end{remark}

\section{Wavelet Transform of FIF}\label{sec:wt}
In this section, the wavelet transform of a FIF is studied through two methods - first using the recursive functional equation given by~\eqref{eq:funceqn} and second using the Fourier transform of FIF given by~\eqref{eq:cor}.

The wavelet transform $W_{\psi}g$ of a function $g \in L_1(\mathbb{R})$ with respect to a suitable wavelet $\psi$ is a function over the half-plane $H=\{(s,t), s,t \in \mathbb{R}, s >0\}$ defined as follows:
\begin{align*}
W_{\psi}g(s,t) = \frac{1}{s}\int g(x)\psi\left(\frac{x-t}{s}\right)  dx.
\end{align*}

\subsection{Wavelet Transform of FIF via functional equation}\label{sec:wt1}

For $k=1,2,\ldots,N$, assume $q_k \in \mbox{Lip}\ \delta,\ 0 < \delta \leq 1$; i.e., for some constant  $K > 0$,  $|q_k(x)-q_k(y)| \leq K |x-y|^{\delta}$ for $x,y \in \mathbb{R}$.  Then, for $k=1,2,\ldots,N$, $q_k \circ L_k^{-1} : I_k \rightarrow \mathbb{R} $ is a function defined on the compact interval $I_k$ and $ q_k \circ L_k^{-1} \in \mbox{Lip}\ \delta$; in fact for $K^*=K N^{\de}$, $|q_k \circ L_k^{-1}(x)-q_k \circ L_k^{-1}(y)| \leq K^* |x-y|^{\delta}$ for $x,y \in \mathbb{R}$.

Let $\psi$ be a wavelet such that $\psi \in L_1(\mathbb{R})$, $\int \psi(x) dx =0$  and $\phi$ defined as $\phi(x)=x^{\delta} \psi(x)$ is also in $L_1(\mathbb{R})$. We also choose the wavelet such that the following conditions are satisfied:
\begin{align}\label{wavcnd}\left.
\begin{array}{l}
(i) \ \hat{\psi}\  \mbox{is real and supp} \hat{\psi} \subset \mathbb{R}^{+}. \\
(ii) \ \mbox{For some}\ r>0, \hat{\psi}(\om) = \om^r + O(\om^{r+1}),\ \om \rightarrow 0^+. \\
(iii) \ \mbox{For each}\ p>0, \hat{\psi}(\om) =  O(\om^{-p}), \  \om \rightarrow \infty.
\end{array}
\right\} \end{align}

It is well known that wavelet transform of any bounded function $g$ with respect to $\psi \in L_1(\mathbb{R})$ is bounded. Since $q_k \circ L_k^{-1} : I_k \rightarrow \mathbb{R} $ are functions defined on the compact interval $I_k, k=1,2,\ldots,N$, $q_k \circ L_k^{-1}$ for $k=1,2,\ldots,n$ are bounded and therefore each $W_{\psi} (q_k \circ L_k^{-1}) $ is also bounded. For a bounded function $g$, if $g(s) = O(|s|^{\de})$ as $s \rightarrow 0$ then $g(s) = O(|s|^{\de})$ for all $s$. Then, from~\cite{holschneider88}, it is observed that the wavelet transform of $q_k \circ L_k^{-1}$ satisfies 
 \begin{align}\label{hols}\left.
 \begin{array}{rll}
 |W_{\psi} (q_k \circ L_k^{-1}) (s,t)|= O(s^{\delta}) &\quad  &\mbox{if} \quad \delta < r \\
 |W_{\psi} (q_k \circ L_k^{-1}) (s,t)|= O(s^m) &\quad  &\mbox{if} \quad \delta > r.
 \end{array}\right\}
 \end{align}
Now, for $(s,t) \in H $, the wavelet transform of FIF is given as
\begin{align*}
W_{\psi} f (s,t) = \frac{1}{s}\int\limits_{I} f(x)\psi\left(\frac{x-t}{s}\right)  dx
= \frac{1}{s} \sum\limits_{k=1}^N \int\limits_{I_k} f(x)\psi\left(\frac{x-t}{s}\right) dx
\end{align*}
Using~\eqref{eq:funceqn} in the above equation, we get, for $(s,t) \in H $,
\begin{align*}
W_{\psi} f (s,t) &= \frac{1}{s} \sum\limits_{k=1}^N \int\limits_{ I_k} \Big[\gm_k f(L_k^{-1}(x)) + q_k(L_k^{-1}(x))\Big]\psi\left(\frac{x-t}{s}\right) dx
 \\&= \frac{1}{s} \sum\limits_{k=1}^N \bigg\{\frac{1}{N} \int\limits_{I} \gm_k f(x) \psi\left(\frac{\frac{x+ k-1}{N}-t}{s}\right)  dx + \int\limits_{I_k}  q_k(L_k^{-1}(x))\psi\left(\frac{x-t}{s}\right) dx \bigg\}\\
&= \sum\limits_{k=1}^N  \Bigg\{\gm_k \ \frac{1}{Ns} \int\limits_{I} f(x)\psi\left(\frac{x -(Nt-k-1)}{Ns}\right)  dx  +  \frac{1}{s}\int\limits_{I_k}  q_k(L_k^{-1}(x))\psi\left(\frac{x-t}{s}\right) dx \Bigg\}\\
& = \sum\limits_{k=1}^N \Bigg\{\gm_k \ W_{\psi}f(Ns,Nt-(k-1)) +  W_{\psi}(q_k \circ L_k^{-1})(s,t) \Bigg\}.
\end{align*}

By induction, for any $n=1,2,\ldots$, we have, for $(s,t) \in H $,
\begin{align}\label{eq:wavtr2}
\lefteqn{W_{\psi} f (s,t)} \nno \\ & = \sum\limits_{k_1,k_2,\ldots,k_n=1}^N \gm_{k_1} \gm_{k_2}\ldots \gm_{k_n} \  W_{\psi}f(N^n s, N^n t-\sum\limits_{j=1}^n N^{n-j}(k_j-1)) \nno  \\ & \quad \mbox{}+ \sum\limits_{j=1}^n \sum\limits_{k_1,k_2,\ldots,k_j=1}^N \gm_{k_1} \gm_{k_2}\ldots \gm_{k_{j-1}}  W_{\psi}(q_{k_j} \circ L_{k_j}^{-1})(N^{j-1} s, N^{j-1} t- \sum\limits_{p=1}^{j-1} N^{j-1-p}(k_p-1))
\end{align}

Defining $\Om_j=N^{\delta}|\gm_j|$ and $\Om =\max \{\Om_j : j=1,\ldots,N\}$ as in~\cite{gang96},  an upper bound on $|W_{\psi} f (s,t)|$ as $s \rightarrow 0$ is obtained in the following theorem:

\begin{theorem}
Let $f$ be a FIF obtained from the interpolation data $\{(x_i,y_i) \in \mathbb{R}^2 : i=0,1,\ldots,N \}$, where $0=x_0<x_1<\ldots<x_N=1$, $x_k -x_{k-1} =\frac{1}{N} $ for $k=1,2,\ldots,N$ and $y_0=y_N=0$. For this, the free parameters $\gm_k$ and functions $q_k, k=1,2,\ldots,N$ in~\eqref{eq:LkFk} satisfy, for some constant $K>0 $ and $0<\delta \leq 1$, $|\gm_k|< \frac{1}{N^{\delta+1}}$ and $|q_k(x)-q_k(y)| \leq K |x-y|^{\delta}$ for $x,y \in \mathbb{R}$ i.e., $q_k \in \mbox{Lip}\ \delta $. Extend $f$ to $\mathbb{R}$ by defining $f(x) = 0 $ if $ x \not \in I =[x_0,x_N]$. The wavelet $\psi$ is chosen such that $\psi \in L_1(\mathbb{R})$, $\int \psi(x) dx =0$,  $\phi$ defined as $\phi(x)=(x^{\delta} \psi(x))$ is also in $L_1(\mathbb{R})$ and the conditions (i),(ii) and (iii) in~\eqref{wavcnd} are satisfied for some $r,p>0$. Then the following holds:
\begin{enumerate}[(a)]
\item If $\de < r$, for $(s,t) \in H $, $|W_{\psi}f(s,t)| \leq  \frac{N K^*}{1-N \Om} |s|^{\delta}$,
 \item $f$ belongs to Lipschitz class of order $\delta$ if $\de < r$.
 \end{enumerate}
\end{theorem}
\begin{proof}
\begin{enumerate}[(a)]
\item  From~\eqref{hols}, $ |W_{\psi} (q_k \circ L_k^{-1}) (s,t)| \leq  K^* s^{\delta}$ if $\delta < r$. Substituting this bound in~\eqref{eq:wavtr2}, we have,
\begin{align}\label{eq:maineq}
|W_{\psi} f (s,t)| & \leq \sum\limits_{k_1,k_2,\ldots,k_n=1}^N |\gm_{k_1}| \ |\gm_{k_2}| \ \ldots \ |\gm_{k_n}| \  |W_{\psi}f(N^n s, N^n t-\sum\limits_{j=1}^n N^{n-j}(k_j-1))| \nno  \\ & \quad \mbox{}+ \sum\limits_{j=1}^n \sum\limits_{k_1,k_2,\ldots,k_{j-1}=1}^N |\gm_{k_0}| \  |\gm_{k_1}| \ \ldots \ |\gm_{k_{j-1}}| \ \sum\limits_{k_j=1}^N K^* (N^{j-1} s)^{\delta} \nno \\
& \leq \sum\limits_{k_1,k_2,\ldots,k_n=1}^N |\gm_{k_1}| \ |\gm_{k_2}| \ \ldots \ |\gm_{k_n}| \  |W_{\psi}f(N^n s, N^n t-\sum\limits_{j=1}^n N^{n-j}(k_j-1))| \nno  \\ & \quad \mbox{}+ K^* N s^{\delta} \sum\limits_{j=1}^n (N \Om)^{j-1}.
\end{align}

Since  $|\gm_k| < \al < \frac{1}{N^{1+\delta}}$ for all $k=1,\ldots,N$, $ N \Om < 1$.
Also, the conditions on $\psi$ tells us that  $|W_{\psi} f |$ is a bounded function. Hence, as $n \rightarrow \infty$, $|W_{\psi}f(s,t)| \leq  \frac{M N }{1-N \Om} |s|^{\delta}$  for all values of $(s,t)$.

\item By Theorem~$2.1.1$ of Chapter~$4$ in~\cite{holschneider95}, $|W_{\psi}f(s,t)| = 0(s^{\delta})$ implies $f$ belongs to Lipschitz class of order $\delta$.

 \end{enumerate}
\qed
\end{proof}

\subsection{Wavelet Transform of FIF via Fourier transform}\label{sec:wt2}

It is well known~\cite{holschneider88} that the wavelet transform of a function $g \in L_1(\mathbb{R})$ with respect to a wavelet $\psi \in L_1(\mathbb{R})$ is also obtained by the following expression:
\begin{align*}
W_{\psi}g(s,t) = \frac{1}{2 \pi} \int \hat{g}(\om ) \overline{\hat{\psi}(s \om)} e^{i t \om} d\om.
\end{align*}

Let $q_k, \ k=1,2,\ldots,N$ in~\eqref{eq:LkFk} be polynomials of degree $m_k$ i.e $q_k(x)=\sum\limits_{r=0}^{m_k} c_{k,r} x^r $. Choose $\psi$ such that $\hat{\psi}(\om) = |\om|_{+}^M \ e^{-\om}$ with $M-1> m=\max\{m_k, \ k=1,2,\ldots,N\}$. Then, for $(s,t) \in H $, the wavelet transform of FIF is given by
\begin{align*}
W_{\psi}f(s,t) &= \frac{1}{2 \pi} \int \hat{f}(\om )\ \overline{\hat{\psi}(s \om)}\ e^{i t \om} d\om
\end{align*}
Using~\eqref{eq:cor} in the above equation, we get,
\begin{align*}
W_{\psi}f(s,t) &= \frac{1}{2 \pi} \int\limits_0^{\infty} \Bigg\{ \sum\limits_{j=1}^{\infty}  \sum\limits_{k_1,k_2,\ldots,k_j=1}^N \gm_{k_1} \gm_{k_2}\ldots \gm_{k_{j-1}} e^{-i \om p_{k_1,k_2,\ldots,k_j}} \times \nno\\ & \quad \mbox{} \times \sum\limits_{r=1}^{m_{k_j}} c_{k_j,r}  \Bigg\{ \bigg[ \frac{r N^{j}}{ \om^2}- \frac{i r (r-1) N^{2j}}{\om^3} - \frac{r (r-1)(r-2)N^{3j}}{\om^4}+  \ldots \nno \\ & \quad \mbox{} - \frac{(-i)^{r+1} N^{rj} r!}{\om^{r+1}} \bigg]e^{-i \om/N^{j}}+ \frac{ (-i)^{r+1} N^{rj} r! }{\om^{r+1}}\Bigg\} (s \om)^M \ e^{- s \om} e^{i t \om} d\om \Bigg\} \\
\end{align*}
Hence, for $(s,t) \in H $,
\begin{align*}
\lefteqn{|W_{\psi}f(s,t)| } \nno \\ & \leq \frac{s^M}{2 \pi} \int\limits_0^{\infty} \sum\limits_{j=1}^{\infty}  \sum\limits_{k_1,k_2,\ldots,k_j=1}^N |\gm_{k_1}| | \gm_{k_2}| \ldots |\gm_{k_{j-1}}| e^{-s \om } \times \nno\\ & \quad \mbox{} \times \sum\limits_{r=1}^{m_{k_j}} |c_{k_j,r}|  \Bigg\{  r N^{j} \om^{M-2}+ r (r-1) N^{2j}\om^{M-3} +   \ldots  + 2 N^{rj} r!\ \om^{M-r-1}\Bigg\}  d\om.
\end{align*}
Fix $0 < s < \infty$. The above integral is divided into $|W_{\psi}f(s,t)|_1$ and $|W_{\psi}f(s,t)|_2$, where
\begin{align}\label{eq:wavtr1}
\lefteqn{|W_{\psi}f(s,t)|_1} \nno \\ & = \frac{s^M}{2 \pi} \int\limits_0^{s} \sum\limits_{j=1}^{\infty}  \sum\limits_{k_1,k_2,\ldots,k_j=1}^N |\gm_{k_1}| | \gm_{k_2}| \ldots |\gm_{k_{j-1}}| e^{-s \om } \times \nno\\ & \quad \mbox{} \times \sum\limits_{r=1}^{m_{k_j}} |c_{k_j,r}|  \Bigg\{  r N^{j} \om^{M-2}+ r (r-1) N^{2j}\om^{M-3} + r (r-1)(r-2)N^{3j}\om^{M-4}+  \ldots \nno \\ & \quad \mbox{} + 2 N^{rj} r!\ \om^{M-r-1}\Bigg\}  d\om 
\end{align}
and
\begin{align}\label{eq:wavtr2}
\lefteqn{|W_{\psi}f(s,t)|_2 } \nno \\& = \frac{s^M}{2 \pi} \int\limits_{s}^{\infty} \sum\limits_{j=1}^{\infty}  \sum\limits_{k_1,k_2,\ldots,k_j=1}^N |\gm_{k_1}| | \gm_{k_2}| \ldots |\gm_{k_{j-1}}| e^{-s \om } \times \nno\\ & \quad \mbox{} \times \sum\limits_{r=1}^{m_{k_j}} |c_{k_j,r}|  \Bigg\{  r N^{j} \om^{M-2}+ r (r-1) N^{2j}\om^{M-3} + r (r-1)(r-2)N^{3j}\om^{M-4}+  \ldots \nno \\ & \quad \mbox{} + 2 N^{rj} r!\ \om^{M-r-1}\Bigg\}  d\om 
\end{align}

Using the regularity notation $ \lambda^{\alpha}$ as in~\cite{holschneider95} and from~\eqref{eq:wavtr1} and ~\eqref{eq:wavtr2}, the regularity of FIF is obtained in the following theorem.

\begin{theorem}
Let $f$ be a FIF obtained from the interpolation data $\{(x_i,y_i) \in \mathbb{R}^2 : i=0,1,\ldots,N \}$,  where $0=x_0<x_1<\ldots<x_N=1$, $x_k -x_{k-1} =\frac{1}{N} $ for $k=1,2,\ldots,N$ and $y_0=y_N=0$. Here, for $k=1,2,\ldots,N$, the functions $q_k$ in~\eqref{eq:LkFk} are polynomials of degree $m_k$ and the free parameters $\gm_k$ satisfy $|\gm_k|< \frac{1}{N^{m+1}}$, where $\ m=\max\{m_k, \ k=1,2,\ldots,N\}$. Extend $f$ to $\mathbb{R}$ by defining $f(x) = 0 $ if $ x \not \in I =[x_0,x_N]$.  Let $\psi$ be such that $\hat{\psi}(\om) = |\om|_{+}^M \ e^{-\om}$ with $M-1>m$. Then the following holds:
\begin{enumerate}[(a)]
\item  $|W_{\psi}f(s,t)| = o(s)$ as $ s \rightarrow 0$,
 \item $f$ is of regularity $M-m$ i.e. $ f \in \lambda(\mathbb{R}), \ \lambda=M-m$.
 \end{enumerate}
\end{theorem}
\begin{proof}
\begin{enumerate}[(a)]
\item  It is observed that, for $p=2,3,\ldots,r+1$,
\begin{align}\label{eq:temp}
 s^M \int\limits_0^s \om^{M-p}e^{-s \om } d \om =0  \quad \mbox{and} \quad  s^M \int\limits_s^{\infty} \om^{M-p}e^{-s \om } d \om =s^{p-1} (M-p)!.
\end{align}
Using~\eqref{eq:temp} in~\eqref{eq:wavtr1} and~\eqref{eq:wavtr2}, we have $|W_{\psi}f(s,t)|_1 = 0 $  and
\begin{align*}
|W_{\psi}f(s,t)|_2 &\leq  \frac{1}{2 \pi} \sum\limits_{j=1}^{\infty}  \sum\limits_{k_1,k_2,\ldots,k_j=1}^N |\gm_{k_1}| | \gm_{k_2}| \ldots |\gm_{k_{j-1}}| \times \nno\\ & \quad \mbox{} \times \sum\limits_{r=1}^{m_{k_j}} |c_{k_j,r}|  \Bigg\{   r N^{j}(M-2)!\ s^{M-1}  +  r (r-1) N^{2j} (M-3)!\ s^{M-2}   +  \ldots \nno\\ & \quad \mbox{}+ 2 N^{rj} r!  (M-r-1)!\ s^{M-r} \Bigg\}.
\end{align*}
Now, as $s \rightarrow \infty$, $\ |W_{\psi}f(s,t)| =|W_{\psi}f(s,t)|_1 = 0 $ and as $\ s \rightarrow 0$, $\ |W_{\psi}f(s,t)| = |W_{\psi}f(s,t)|_2 \leq C s^{M-m}$ for a suitable constant $C$.    Hence,  $\ |W_{\psi}f(s,t)| = o(s)$ as $ s \rightarrow 0$.
\item By Theorem~$2.1.1$ of Chapter~$4$ in~\cite{holschneider95}, $|W_{\psi}f(s,t)| = o(s)$ implies $f$ is of regularity $M-m$ i.e. $ f \in \lambda(\mathbb{R}),\ \lambda=M-m$.
 \end{enumerate}
\qed
\end{proof}

{\sc \textbf{Acknowledgement}}

  I am extremely thankful to Prof. Ajit Iqbal Singh for all the helpful discussions we had throughout the course of preparing this manuscript. I am also thankful to NBHM for postdoctoral research grant.


\end{document}